\documentclass[12pt]{article}%
\usepackage{amsfonts}
\usepackage{amsmath}
\usepackage{amssymb}
\usepackage{graphicx}%
\providecommand{\U}[1]{\protect\rule{.1in}{.1in}}
\newtheorem{theorem}{Theorem}

\newtheorem{corollary}[theorem]{Corollary}

\newtheorem{lemma}[theorem]{Lemma}

\newtheorem{proposition}[theorem]{Proposition}

\newenvironment{proof}[1][Proof]{\noindent\textbf{#1.} }{\ \rule{0.5em}{0.5em}}
\begin{document}

\title{On lattice homomorphisms in Riesz spaces }
\author{Elmiloud Chil, Fateh Mekdour\\Elmiloud Chil University of Tunis \\Institut pr\'{e}paratoire aux \'{e}tudes d'ingenieurs de Tunis\\2 Rue Jawaher lel Nehrou Monflery 1008 TUNISIA,\\Elmiloud.chil@ipeit.rnu.tn\\Fateh Mekdour L.A.T.A.O. Faculty of sciences of Tunis University\\ELManar Compus universitaire Elmanar TUNISIA,\\fatehmekdour@gmail.com}
\date{}
\maketitle

\begin{abstract}
In this paper, we study the connextion between lattice and Riesz homomorphisms
in Riesz spaces. We prove, under a certain condition, that any lattice
homomorphism on Riesz space is a Riesz homomorphism. This fits in the type of
results by Mena and Roth [7], thanh [9], Lochan and Strauss [5] and Ercan and
Wickstead [2].

\end{abstract}



\section{Introduction and Preliminaries}

Let $E$ and $F$ be two Riesz spaces. A mapping $T:E\rightarrow F$ is said to
be a \textit{lattice homomorphism} after Mena and Roth in [7] whenever
\[
T(x\vee y)=Tx\vee Ty\text{ and }T(x\wedge y)=Tx\wedge Ty\text{, for all
}x,y\in E\text{.}%
\]
A linear homomorphism is called a \textit{Riesz homomorphism}. We will look at
the following questions: When is all lattice homomorphism Riesz homomorphism?
This problem can be treated in a different manner, depending on domains and
co-domains on which maps acts. In the last decade, several authors have been
studied this problem on different Archimedean Riesz spaces. The study of the
relationship between lattice and Riesz homomorphism was really inaugurated in
1978 thanks to the fundamental work of Mena and Roth [7], who were interested
essentially for the setting of the Riesz spaces of real valued continuous
functions defined on a compact Hausdorff spaces. They proved that if $X$ and
$Y$ are compact Hausdorff spaces and $T:C(X)\rightarrow C(Y)$ is a lattice
homomorphism such that $T(\lambda1)=\lambda T(1)$ for all $\lambda\in%
\mathbb{R}
$, then $T$ is linear. Later, many authors interested in this problem. In [9],
Thanh generalized Mena and Roth's result to the case when $X$ and $Y$ are real
compact spaces. For another generalization by Lochan and Strauss see [5]. So
far the best results, in this field, are duo to Ercan and Wickstead in [2].
They showed from the theorem of Mana and Roth by using the Kakutani
representation theorem, that if $E$ and $F$ are uniformly complete Archimedean
Riesz spaces with weak order units $e_{1}\in E$ and $e_{2}\in F$ and if
$T:E\rightarrow F$ is a lattice homomorphism such that $T(\lambda
e_{1})=\lambda e_{2}$ for all $\lambda\in%
\mathbb{R}
$, then $T$ is linear. As an application they gave a corresponding result for
the case where $E$ and $F$ are two uniformly complete Archimedean Riesz spaces
with disjoint complete systems of projections as an example for two $\sigma
$-Dedekind complete Riesz spaces. It will be noted that the proofs of the last
results are heavily based on the uniform structure of the spaces, hence it can
not be expected that it can be adapted to the Riesz space case. Surprisingly
enough, to the best of our knowledge, no attention has been paid in the
literature of this problem for the general case of Riesz spaces. It seems that
the main reason for this, it lies in the fact of the topological richness of
the uniformly complete structure of Riesz spaces with weak order units which
gives a particular interest to lattice homomorphisms on those spaces. It seems
natural therefore to ask what happenings in the general case of Riesz spaces?
What about weaker condition under which a lattice homomorphism defined on a
Riesz space is linear? The above results tell us what is happening in some
cases, but there are certainly gaps in the general case. The complications are
partly due to the delicate behavior of lattice homomorphisms, partly to the
topological poverty of the domain. If we focus our attention to find a manner
for extending the proofs of aforementioned results from an Archimedean Riesz
space to its uniformly completion, then we would hope to establish more. In
the present paper, as indicated by its title, we take a look at this problem.
Our main goal is to prove that, in the results of Mena and Roth [7], Thanh
[9], Lochan and Strauss [5] and Ercan and Wickstead [2], the assumption of the
uniform completeness condition on the domain of mapping is superfluous. As
applications we give some generalization for those results. Actually, we
provide not only new results but also new techniques, which we think that are
useful additions to the literature. Our results here are to be compared with
the corresponding ones in the last mentioned papers, in this respect we find
that the lattice structure of maps is quite natural in this setting, and even
that it has some advantages, as the key results are obtained more directly constructive.

We take it for granted that the reader is familiar with notions of Riesz
spaces (or vector lattices) and operators between them. For terminology,
notations and concepts that are not explained in this paper, one can refer to
the standard monographs [1,6,8].

In order to avoid unnecessary repetition we will assume throughout the paper
that all Riesz spaces under consideration are Archimedean.

A Riesz space is called \textit{Dedekind complete} whenever every nonempty
subset that is bounded from above has a supremum. Similarly, a Riesz space is
said to be $\sigma$\textit{-Dedekind complete} whenever every countable subset
that is bounded from above has a supremum. Let us say that a vector subspace
$G$ of a Riesz space $E$ is \textit{majorizing} $E$ whenever for each $x\in E$
there exists some $y\in G$ with $x\leq y$.(or, equivalently whenever for each
$x\in E$ there exists some $y\in G$ with $y\leq x$)$.$ A Dedekind complete
Riesz space $L$ is said to be a \textit{Dedekind completion} of the Riesz
space $E$ whenever $E$ is Riesz isomorphic to an order dense majorizing Riesz
subspace of $L$ (which we identify with $E$). It is a classical result that
every Archimedean Riesz space $E$ has a Dedekind completion, which we shall
denoted by $E^{\delta}$.

The relatively uniform topology on Riesz spaces plays a key role in the
context of this work. Let us recall the definition of the \textit{relatively
uniform convergence}. Let $E$ be an Archimedean Riesz space and an element
$u\in E^{+}$. A sequence $(x_{n})_{n}$ of elements of $E$ \textit{converges
}$u$\textit{-uniformly} to the element $x\in E$ whenever, for every
$\epsilon\geq0$, there exists a natural number $N_{\epsilon}$ such that
$|x_{n}-x|\leq\epsilon u$ holds for all all $n\geq N_{\epsilon}$. This will be
denoted $x_{n}\rightarrow x(r.u)$. The element $u$ is called the
\textit{regulator of the convergence}. The sequence $(x_{n})_{n}$
\textit{converges relatively uniformly} to $x\in E$, whenever $x_{n}%
\rightarrow x(u)$ for some $u\in E^{+}$. We shall write $x_{n}\rightarrow
x(r.u)$ if we do not want to specify the regulator. \textit{Relatively uniform
limits} are unique if and only if $E$ is Archimedean [6, Theorem 63.2]. A
nonempty subset $D$ of $E$ is said to be \textit{relatively uniformly closed}
if every relatively uniformly convergent sequence in $D$ has its limit also in
$D$. We emphasize that the regulator does need not to be an element of $D$.
The relatively uniformly closed subsets are the closed sets of a topology in
$E$, namely, the relatively uniform topology. The closure, with respect to the
relatively uniform topology of the Riesz space $E$ in its Dedekind completion
is a uniformly complete vector lattice, denoted by $E^{ru}$ and referred to as
the uniform completion of $E$ (see [6]). Thus $E$ is an order dense majorizing
Riesz subspace of $E^{ru}$.

The notion of \textit{relatively uniform Cauchy sequences} is defined in the
obvious way. A Riesz space $E$ is said to be \textit{relatively uniformly
complete} whenever every relatively uniformly Cauchy sequence has a (unique)
limit. For more details we refer the reader to [6].

\section{Main results}

In order to reach our main result we need some prerequisites.

\begin{lemma}
Let $E$ be a Riesz space and $E^{ru}$ its uniformly complete then for every
$x\in E^{ru}$ there exists a sequence $(x_{n})_{n}\in E$ such that
$x_{n}\nearrow x(r.u)$.
\end{lemma}

\begin{proof}
Let $0\neq x\in E^{ru}$. Since $E$ is majorizing in $E^{ru}$, there exists
$0<y\in E$ such that%
\[
0<\left\vert x\right\vert \leq y.
\]
Now let $E_{y}$ be the principal ideal generated by $y$ in $E,$ and let
$E_{y}^{ru}$ be the principal ideal generated by $y$ in $E^{ru}.$ By the
Kakutani theorem [8, Theorem 2.1.3], $E_{y}^{ru}$ is order isomorphic to a
$C(X)$ for a compact Hausdorff space $X$ and so $E_{y}$ is uniformly dense in
$C(X)=E_{y}^{ru}$. Now by the fact that in $C(X)$ the relatively uniform
convergence and uniform convergence of sequences coincide, the relatively
uniform closure of a set is the familiar uniform closure and the relatively
uniform topology is the norm topology with respect to the uniform norm. It
immediately follows that the pseudo closure and the closure of any set with
respect to this topology are the same. Since $x\in E_{y}^{ru}$, there exists a
sequence $(x_{n})_{n}\in E_{y}$ such that $x_{n}\nearrow x(r.u)$.
\end{proof}

The following results deals with prime ideal in Riesz space which plays an
important role in our approach. A prime ideal $P$ of a Riesz space $E$ is a
nonempty proper lattice subset of $E$ (not necessarily a vector subspace) that
containing with any element all smaller ones and with $x\in P$ or $y\in P$
whenever $x\wedge y\in P$. We say that a prime ideal $P$ in $C(X)$ is
associated with a point $x\in X$ if $g\in P$ whenever $f\in P$ and
$g(x)<f(x).$ In [4] Kaplansky proved that for $X$ compact, every prime ideal
in $C(X)$ is associated with some point of $X$ and this point is unique if $X$
is also Hausdorff. In the same paper, he proved also that if $P\subset Q$
where $P$ and $Q$ are prime ideals in $C(X),$ $X$ is a compact Hausdorff space
then $P$ and $Q$ are associated with the same point. For more details about
prime ideals in $C(X)$ see [4].

The following result will be of great use next.

\begin{proposition}
Let $E$ be a Riesz space and $P$ be a prime ideal of $E$ then $P^{ru}$, the
uniform completion of $P$, is a prime ideal of $E^{ru}$.
\end{proposition}

\begin{proof}
First we prove that $P^{ru}$ is a proper set of $E^{ru}$. Suppose that
$P^{ru}=E^{ru}$ then $E\subset P^{ru}$ on the other hand for every $x\in
E\subset P^{ru}$ there exists $y\in P$ such that $x\leq y$ (because $P$ is
\textit{majorizing} $P^{ru}$) which implies that $x\in P$ as $P$ is an ideal
in $E$ which implies that $E\subset P$ then $P=E$. This contradict with $P$
proper. Consequently $P^{ru}$ is a proper set of $E^{ru}$. Secondly it is not
hard to prove that $P^{ru}$ is a lattice ideal of $E^{ru}$. It remains to show
the prime property of $P^{ru}$. To do this let $x,y\in E^{ru}$ such that
$x\wedge y\in P^{ru}$ we claim that $x\in P^{ru}$ or $y\in P^{ru}$. By using
the preceding Lemma there exist two sequences $(x_{n})_{n}$,$(y_{n})_{n}\in E$
such that $x_{n}\nearrow x(r.u)$ and $y_{n}\nearrow y(r.u)$ then $x_{n}\wedge
y_{n}\nearrow x\wedge y$. Now by using the majorizing property of $P$ in
$P^{ru}$ there exists $z\in P$ such that $x\wedge y\leq z$ and thus
$x_{n}\wedge y_{n}\leq x\wedge y\leq z$ then $x_{n}\wedge y_{n}\leq z$.
Therefore $x_{n}\wedge y_{n}\in P$ which implies that $x_{n}\in P$ or
$y_{n}\in P$. Thus $x\in P^{ru}$ or $y\in P^{ru}$ and we are done.
\end{proof}

The main result of this paper is strongly based on the following proposition.

\begin{proposition}
Let $E$ be a uniformly dense Riesz subspace of $C(X)$ where $X$ is a compact
Haudorff space and $P$ a prime ideal. Then $P$ and $P^{ru}$ are associated
with the same point $x\in X$.
\end{proposition}

\begin{proof}
By the preceding proposition $P^{ru}$ is a prime ideal of $E^{ru}=C(X)$ then
from [4] there exist $x\in X$ such that $P^{ru}$ is associated with $x$. We
claim that $P$ is also associated with $x$. To see this, let $g\in E$ and
$f\in P$ such that $g(x)<f(x)$ we claim that $g\in P$. By the fact that
$E\subset E^{ru}$ and $P\subset P^{ru}$, we have $g\in E^{ru}$, $f\in P^{ru}$
so by $g(x)<f(x)$ we have $g\in P^{ru}$ (because $P^{ru}$ is associated with
$x$) then there exists $h\in P$ such that $g\leq h$ which implies $g\in P$
(because $P$ is an ideal of $E$). Which is the desired result.
\end{proof}

We arrive at this point to a one of the central results of this paper in which
we provide a representation of every real lattice homomorphism on a uniformly
dense Riesz subspace of $C(X)$ as evaluation at some point of $X$. We note
that the proof of the corresponding result in the case of $C(X)$ (see [7]) is
heavily based on the structure of $X$ and hence it can not be expected that it
can be adapted directly to the general Riesz space case. The complications are
essentially due to the topological poverty of the domain. Our curiosity about
this problem was in part initiated by the fact that after reading clearly the
details of proofs of aforementioned results we were certain that with similar
techniques we can do better. A crucial piece of informations that we need from
the richness of the uniform completion structure are came from its one of
prime ideals proved in last results. We give a generalization of results in
[2,7] to uniformly dense Riesz subspaces of $C(X)$ by adding some facts to
Mena Roth's proofs. To clarify our contribution we give a details proofs for
the two following results.

\begin{proposition}
Let $E$ be a uniformly dense Riesz subspace of $C(X)$ where $X$ is a compact
Haudorff space such that $1\in E$. Let $\phi:E\rightarrow%
\mathbb{R}
$ be a lattice homomorphism satisfying $\phi(\lambda)=\lambda$ for all
$\lambda\in%
\mathbb{R}
$. Then there exist $x\in X$ such that $\phi=\delta_{x}$, the point evaluation
map defined for $f\in C(X)$ by $\delta_{x}(f)=f(x)$.
\end{proposition}

\begin{proof}
Let $I_{a}^{-}=\left\{  r\in%
\mathbb{R}
:r\leq a\right\}  $ for $a\in%
\mathbb{R}
$. As it is mentioned in [7] $\left\{  I_{a}^{-}:a\in%
\mathbb{R}
\right\}  $ and $\left\{  \phi^{-1}(I_{a}^{-}):a\in%
\mathbb{R}
\right\}  $ are chains of prime ideals in $%
\mathbb{R}
$ and $E$, respectively$.$ Now by using proposition 1 $\left\{  (\phi
^{-1}(I_{a}^{-}))^{ru}:a\in%
\mathbb{R}
\right\}  $ is a chain of prime ideals in $E^{ru}=C(X)$. Then by the preceding
proposition $(\phi^{-1}(I_{a}^{-}))^{ru}$ and $(\phi^{-1}(I_{a}^{-}))$ are
associated with the same point $x\in X$ for all $a\in%
\mathbb{R}
.$ We claim that $\phi=\delta_{x}$. To do this let $f\in E$. By using the fact
that $f\in\phi^{-1}(I_{\phi(f)}^{-})$ and $x$ is an associated point for
$\phi^{-1}(I_{\phi(f)}^{-})$ we have $r\in\phi^{-1}(I_{\phi(f)}^{-})$ for
every $r<f(x)$. It follows that $r\leq\phi(f)$ for every $r<f(x)$ which
implies that $f(x)\leq\phi(f)$ (by limits if $r\rightarrow f(x)$). Also by the
fact that $r\in\phi^{-1}(I_{r}^{-})$ and $x$ is an associated point for
$\phi^{-1}(I_{r}^{-})$ we have $f\in\phi^{-1}(I_{r}^{-})$ for every $r>f(x)$.
It follows that $\phi(f)\leq r$ for every $f(x)<r$ which implies that
$\phi(f)\leq f(x)$ (by limits if $r\rightarrow f(x)$). Therefore
$\phi(f)=f(x)=\delta_{x}(f)$ for every $f\in E$ which implies that
$\phi=\delta_{x}$.
\end{proof}

Now we can give a generalization of Mena and Roth, thanh, Lochan and Strauss,
Ercan and Wickstead's approaches [2,5,7,9] for Riesz spaces with strong order
unit as follows:

\begin{theorem}
Let $E$ be a Riesz space with strong order unit $e$ and let $F$ be a Riesz
space. If $T$ is a lattice homomorphism from $E$ into $F$ such that $T(\lambda
e)=\lambda T(e)$ for each $\lambda\in%
\mathbb{R}
$ then $T$ is linear (Riesz homomorphism).
\end{theorem}

\begin{proof}
First we claim that $T(E)\subset F_{T(e)}$, the principal ideal generated by
$T(e)$ in $F$. Take $y\in T(E)$, then there exists $x\in E$ such that
$y=T(x)$. Now, from the fact that $e$ is a strong order unit there exists
$\lambda\in%
\mathbb{R}
^{+}$ such that $-\lambda e\leq x\leq\lambda e$ since $T$ is an increasing map
we obtain $T(-\lambda e)=-\lambda T(e)\leq T(x)\leq T(\lambda e)=\lambda
T(e)$. Finally, $\left\vert T(x)\right\vert \leq\lambda T(e)$ thus $T(x)\in
F_{T(e)}$ which implies that $T(E)\subset F_{T(e)}$. On the other hand, we can
regard $T$ as a map from $E$ into $F_{T(e)}^{ru}$, the uniform completion of
$F_{T(e)}$. According to the Kakutani theorem [8,Theorem 2.1.3] there exist
two compact Hausdorff spaces $X$ and $Y$ such that $E^{ru}$ and $(F_{T(e)}%
)^{ru}$ can be identified with $C(X)$ and $C(Y)$, respectively$.$ Now let
$y\in Y$ then the map $\delta_{y}\circ T:E\rightarrow%
\mathbb{R}
$ defined by $\delta_{y}\circ T(f)=T(f)(y)$ for all $f\in E$ is a lattice
homomorphism fixing the constants and thus by the above proposition there
exists a unique $x\in X$ with
\[
\delta_{y}\circ T=\delta_{x}%
\]
which implies that $T(f)(y)=f(x)$ for all $f\in E$. Noting $y=\varphi(x)$
thus
\[
T(f)=f\circ\varphi\text{.}%
\]
Consequently $T$ is linear which gives the desired result.
\end{proof}

In the next result we prove that the condition of the Dedekind $\sigma
$-completeness in Ercan and Wickstead result's [2, corollary 8] can be removed
as follows:

\begin{theorem}
Let $E$ and $F$ be Riesz spaces. If $T$ is a lattice homomorphism from $E$
into $F$ such that $T(\lambda x)=\lambda T(x)$ for each $\lambda\in%
\mathbb{R}
$ and $x\in E$ then $T$ is linear (Riesz homomorphism).
\end{theorem}

\begin{proof}
It is suffices to prove that $T$ is additive. To do this let $x,y\in E$ and
put $e=\left\vert x\right\vert +\left\vert y\right\vert $. Denote by $E_{e}$
the principal ideal generated by $e$ in $E$ and $T_{e}$ the restriction of $T$
to $E_{e}$. Therefore $E_{e}$ is a Riesz space with strong order unit $e$ and
$T_{e}:E_{e}\rightarrow F$ is a lattice homomorphism which satisfies
$T_{e}(\lambda e)=\lambda T_{e}(e)$. According to the preceding Theorem
$T_{e}$ is linear. Since $x,y,x+y\in E_{e}$ it follows that $T(x+y)=T_{e}%
(x+y)=T_{e}(x)+T_{e}(y)=T(x)+T(y)$, which gives the desired result.
\end{proof}

Now we are able to announce the main result of this paper in which we give a
Riesz space version of Mena and Roth, thanh, Lochan and Strauss, Ercan and
Wickstead's approaches [2,5,7,9] as follows. For the proof we use a similar
techniques that proves lemma 1 in [2] with some addition and more details.

\begin{theorem}
Let $E$ and $F$ be a Riesz spaces with weak order units $e$ and $f$,
respectively. If $T$ is a lattice homomorphism from $E$ into $F$ satisfying
$T(\lambda e)=\lambda f$ for each $\lambda\in%
\mathbb{R}
$ then $T$ is a Riesz homomorphism (linear).
\end{theorem}

\begin{proof}
Let $E_{e}$ and $F_{f}$ be the principal order ideals generated by $e$ in $E$
and $f$ in $F$, respectively. As proved in theorem 1 we have $T(E_{e})\subset
F_{f}$. Denote $T_{e}$ the restriction of $T$ to $E_{e}$, that is $T_{e}%
:E_{e}\rightarrow F_{f}$ defined by $T_{e}(x)=T(x)$. By applying theorem 1
$T_{e}$ is a Riesz homomorphism (linear). On the other hand $T(e)=f$ is a weak
order unit of $F$ so for every $x\geq0$ ($T(x)\geq0$ because $T$ is positive)
we have
\[
T(x)=\sup(T(x)\wedge nT(e))=\sup(T(x\wedge ne))=\sup(T_{e}((x\wedge ne)).
\]
We claim that $T$ is additive in $E^{+}$. Let $x,y\in E^{+}$, then the
inequality $(x+y)\wedge ne\leq x\wedge ne+y\wedge ne$ for all $n\in%
\mathbb{N}
$ shows that
\[
T_{e}((x+y)\wedge ne)\leq T_{e}(x\wedge ne+y\wedge ne)=T_{e}(x\wedge
ne)+T_{e}(y\wedge ne)\leq T(x)+T(y)
\]
and hence
\[
T(x+y)\leq T(x)+T(y)\text{.}%
\]
On the other hand, the inequality $x\wedge ne+y\wedge me\leq x+y$ for all
$n,m\in%
\mathbb{N}
$ implies that
\[
T(x\wedge ne+y\wedge me)=T_{e}(x\wedge ne+y\wedge me)=T_{e}(x\wedge
ne)+T_{e}(y\wedge me)\leq T(x+y)\text{.}%
\]
Now by taking supremum on $n,m\in%
\mathbb{N}
$ we obtain
\[
T(x)+T(y)\leq T(x+y)
\]
therefore
\[
T(x+y)=T(x)+T(y)
\]
for all $x,y\in E^{+}$. Let $S$ be the restriction of $T$ to $E^{+}$. From the
Kantorovich theorem (see [1, Theorem 1.7]) $S$ extends uniquely to a positive
operator $S$ from $E$ into $F$ by putting
\[
S(x)=S(x^{+})-S(x^{-})=T(x^{+})-T(x^{-})
\]
for all $x\in E$. Next, we claim that $T(-x)=-T(x)$ for all $x\in E^{+}$.
Since $-T(-x)\geq0$ (because $T$ is positive) we have
\begin{align*}
-T(-x)  &  =\sup((-T(-x))\wedge nT(e))=\sup(-(T(-x)\vee(-ne))\\
&  =\sup(-(T_{e}(-x)\vee(-ne))=\sup(T_{e}(-((-x)\vee(-ne)))\\
&  =\sup(T_{e}(x\wedge ne))=T(x)\text{.}%
\end{align*}
for all $x\in E^{+}.$ Consequently,
\[
T(x)^{-}=(-T(x))\vee0=-T(x\wedge0)=-T(-x^{-})=T(x^{-})
\]
for all $x\in E$. Hence
\[
T(x)=T(x)^{+}-T(x)^{-}=T(x^{+})-T(x^{-})=S(x).
\]
Which implies that $T$ is linear and so $T$ is a Riesz homomorphism and the
proof is finished.
\end{proof}

As an immediate application to the above theorem we give the following result.

\begin{corollary}
Let $E$ and $F$ be Riesz spaces with a weak order unit $e$ of $E$. If $T$ is a
lattice isomorphism from $E$ into $F$ such that $T(\lambda e)=\lambda T(e)$
for each $\lambda\in%
\mathbb{R}
$ then $T$ is a Riesz isomorphism (linear).
\end{corollary}

\begin{proof}
It is sufficient to prove that $Te$ is a weak order unit of $F$. To see this,
let $0\leq y\in F$ and $y\wedge T(e)=0$ choose $0\leq x\in E$ with $T(x)=y$
then
\[
0=y\wedge T(e)=T(x)\wedge T(e)=T(x\wedge e)\text{.}
\]
This implies that $x\wedge e=0$ (because $T$ is an isomorphism). Then $x=0$
because $e$ is a weak order unit of $E$, so $y=0$. Hence $Te$ is a weak order
unit of $F$. Now, from the above theorem $T$ is a Riesz isomorphism.
\end{proof}

At the end of this paper, we give an important Banach-stone type theorem for
non linear homomorphisms between lattices of Lipschitz functions on metric
spaces (which are, in general, not uniformly complete spaces). For a metric
space $(X,d)$ we denote by $Lip(X)$, the unital vector lattice of all
lipschitz real function on $X$. In the following we show that the unital
vector lattice structure of $Lip(X)$ determines the lipschitzian structure of
a complete metric space $X$. Completeness can not be avoided here, since every
metric space has the same lipschitz functions as its completion.

\begin{theorem}
Let $(X,d)$ and $(Y,d^{\prime})$ be complete metric spaces and
$T:Lip(X)\rightarrow Lip(Y)$ be a lattice isomorphism with $T(\lambda
1)=\lambda T(1)$ for each $\lambda\in%
\mathbb{R}
$. Then $X$ and $Y$ are lipschitz homeomorphic.
\end{theorem}

\begin{proof}
It is plain that $1$ is a weak order unit of $Lip(X)$ so from Corollary 1 $T$
is a Riesz isomorphism. Now, According to [3, Theorem 3.10, ] $X$ and $Y$ are
lipschitz homeomorphic.
\end{proof}

\end{document}